\documentclass{amsart}

\usepackage{amssymb, amsmath}
\usepackage{mathrsfs}
\usepackage{amscd}
\usepackage{verbatim}

\usepackage[colorlinks,linkcolor={blue},citecolor={blue},urlcolor={red},]{hyperref}

\theoremstyle{plain}
\newtheorem{theorem}{Theorem}
\theoremstyle{remark}
\newtheorem{remark}[theorem]{Remark}
\newtheorem{example}[theorem]{Example}
\theoremstyle{plain}
\newtheorem{corollary}[theorem]{Corollary}
\newtheorem{lemma}[theorem]{Lemma}
\newtheorem{proposition}[theorem]{Proposition}

\newtheorem{problem}[theorem]{Problem}

\def\N{{\mathbb N}}

\def\R{{\mathbb R}}

\newcommand{\E}{{\mathbb E}}
\renewcommand{\P}{{\mathbb P}}
\newcommand{\F}{{\mathcal F}}

\newcommand{\e}{\varepsilon}

\renewcommand{\O}{\Omega}

\newcommand{\beq}{\begin{equation}}
\newcommand{\eeq}{\end{equation}}
\newcommand{\bal}{\begin{aligned}}
\newcommand{\eal}{\end{aligned}}
\newcommand{\ben}{\begin{enumerate}}
\newcommand{\een}{\end{enumerate}}
\newcommand{\bit}{\begin{itemize}}
\newcommand{\eit}{\end{itemize}}

\newcommand{\bth}{\begin{theorem}}
\renewcommand{\eth}{\end{theorem}}
\newcommand{\bpr}{\begin{proposition}}
\newcommand{\epr}{\end{proposition}}
\newcommand{\ble}{\begin{lemma}}
\newcommand{\ele}{\end{lemma}}
\newcommand{\bpf}{\begin{proof}}
\newcommand{\epf}{\end{proof}}
\newcommand{\bex}{\begin{example}}
\newcommand{\eex}{\end{example}}
\newcommand{\bre}{\begin{example}}
\newcommand{\ere}{\end{example}}

\newcommand{\calL}{{\mathcal L}}

\newcommand{\one}{{{\bf 1}}}

\newcommand{\limn}{\lim_{n\to\infty}}

\newcommand{\inv}[1]{\frac{1}{#1}}

\begin{document}

\title
[tangent martingale difference sequences]{Some remarks on tangent martingale
difference sequences in $L^1$-spaces}

\address{Delft Institute of Applied Mathematics\\
Delft University of Technology \\ P.O. Box 5031\\ 2600 GA Delft\\The
Netherlands}

\author{Sonja Cox}
\email{S.G.Cox@tudelft.nl, sonja.cox@gmail.com}

\author{Mark Veraar}
\email{M.C.Veraar@tudelft.nl, mark@profsonline.nl}

\thanks{The second named author is supported by the Netherlands Organisation for
Scientific Research (NWO) 639.032.201 and the Research Training
Network MRTN-CT-2004-511953}



\date\today

\begin{abstract}
Let $X$ be a Banach space. Suppose that for all $p\in (1, \infty)$ a
constant $C_{p,X}$ depending only on $X$ and $p$ exists such that
for any two $X$-valued martingales $f$ and $g$ with tangent
martingale difference sequences one has
\[\E\|f\|^p \leq C_{p,X} \E\|g\|^p \ \ \ \ \ \  (*).\]
This property is equivalent to the UMD condition. In fact, it is
still equivalent to the UMD condition if in addition one demands
that either $f$ or $g$ satisfy the so-called (CI) condition.
However, for some applications it suffices to assume that $(*)$
holds whenever $g$ satisfies the (CI) condition. We show that the
class of Banach spaces for which $(*)$ holds whenever only $g$
satisfies the (CI) condition is more general than the class of UMD
spaces, in particular it includes the space $L^1$. We state several
problems related to $(*)$ and other decoupling inequalities.
\end{abstract}


\maketitle

\section{Introduction}
Let $(\O, {\mathcal A}, \P)$ be a complete probability space. Let
$X$ be a Banach space and let $(\mathcal{F}_n)_{n\geq 0}$ be a
filtration. The $(\mathcal{F}_n)_{n\geq 1}$-adapted sequences of
$X$-valued random variables $(d_n)_{n\geq 1}$ and $(e_n)_{n\geq 1}$
are called \emph{tangent} if for every $n=1,2,\ldots$ and every
$A\in \mathcal{B}(X)$
\[\begin{aligned}
\mathbb{E}(1_{\{d_n\in
A\}}|\mathcal{F}_{n-1})=\mathbb{E}(1_{\{e_n\in
A\}}|\mathcal{F}_{n-1}).
\end{aligned}\]

An $(\mathcal{F}_n)_{n\geq 1}$-adapted sequence of $X$-valued random
variables $(e_n)_{n\geq 1}$ is said to satisfy the \emph{(CI)
condition} if there exists a $\sigma$-field
$\mathcal{G}\subset\mathcal{F}=\sigma(\cup_{n\geq 0}\mathcal{F}_n)$
such that for every $n\in\mathbb{N}$ and every $A\in\mathcal{B}(X)$
\[\begin{aligned}
\mathbb{E}(1_{\{e_n\in
A\}}|\mathcal{F}_{n-1})=\mathbb{E}(1_{\{e_n\in A\}}|\mathcal{G})
\end{aligned}\]
and if moreover $(e_n)_{n\geq 1}$ is a sequence of
$\mathcal{G}$-conditionally independent random variables, i.e.\ for
every $n=1,2,\ldots$ and every $A_1,\ldots,A_n\in\mathcal{B}(X)$ we
have
\[\begin{aligned}
\mathbb{E}(1_{\{ e_1\in A_1\}}\cdot \ldots \cdot 1_{\{e_n\in
A_n\}}|\mathcal{G})=\mathbb{E}(1_{\{e_1\in A_1\}}|\mathcal{G})\cdot
\ldots \cdot \mathbb{E}(1_{\{ e_n\in A_n\}}|\mathcal{G}).
\end{aligned}\]
The above concepts were introduced by Kwapie{\'n} and Woyczy{\'n}ski
in \cite{KwWo1}. For details on the subject we refer to the
monographs \cite{delaPG,KwWo} and the references therein. It is also
shown there that for every sequence $(d_n)_{n\geq 1}$ of
$(\F_n)_{n\geq 1}$-adapted random variables there exists another
sequence $(e_n)_{n\geq 1}$ (on a possibly enlarged probability
space) which is tangent to $(d_n)_{n\geq 1}$ and satisfies the (CI)
condition. One easily checks that this sequence is unique in law.
The sequence $(e_n)_{n\geq 1}$ is usually referred to as the {\em
decoupled tangent sequence}.

\begin{example}
Let $(\xi_n)_{n\geq 1}$ be an $(\F_n)_{n\geq 1}$-adapted sequence of
real valued random variables. Let $(\widehat{\xi}_n)_{n\geq 1}$ be
copy of $(\xi_n)_{n\geq 1}$ independent of $\F_{\infty}$. Let
$(v_n)_{n\geq 1}$ be an $(\F_n)_{n\geq 0}$-predictable sequences of
$X$-valued random variables, i.e. each $v_n$ is $\F_{n-1}$
measurable. For $n\geq 1$, define $d_n = \xi_n v_n$ and $e_n =
\widehat{\xi}_n v_n$. Then $(d_n)_{n\geq 1}$ and $(e_n)_{n\geq 1}$
are tangent and $(e_n)_{n\geq 1}$ satisfies the (CI) condition with
$\mathcal{G} = \F_\infty$.
\end{example}

For convenience we will assume below that all martingales start at
zero. This is not really a restriction as can be seen as in
\cite{Burkholder:Geom}.

Recall that a Banach space $X$ is a {\em UMD space} if for some
(equivalently, for all) $p\in (1,\infty)$ there exists a constant
$\beta_{p,X}\ge 1$ such that for every martingale difference
sequence $(d_n)_{n\geq 1}$ in $L^p(\O;X)$, and every $\{-1,
1\}$-valued sequence $(\e_n)_{n\geq 1}$ we have
\begin{equation}\label{eq:UMD}
\Bigl(\E\Bigl\|\sum_{n=1}^N \e_n d_n \Bigr\|^p\Bigr)^\frac1p \leq
\beta_{p,X} \, \Bigl(\E\Bigl\|\sum_{n=1}^N
d_n\Bigr\|^p\Bigr)^\frac1p, \ \ N\geq 1.
\end{equation}
One can show that UMD spaces are reflexive. Examples of UMD spaces
are all Hilbert spaces and the spaces $L^p(S)$ for all $1<p<\infty$
and $\sigma$-finite measure spaces $(S,\Sigma,\mu)$. If $X$ is a UMD
space, then $L^p(S;X)$ is a UMD space for $1<p<\infty$. For an
overview of the theory of UMD spaces we refer the reader to
\cite{Bu3} and references given therein.

The UMD property can also be characterized using a randomization of
the martingale difference sequence. This has been considered in
\cite{Garling:RMTI} by Garling. One has that $X$ is a UMD space if
and only if for some (equivalently, for all) $p\in (1,\infty)$ there
exists a constant $C_p\ge 1$ such that for every martingale
difference sequence $(d_n)_{n\geq 1}$ in $L^p(\O;X)$ we have
\begin{equation}\label{eq:rUMD}
C_p^{-1}\Bigl(\E\Bigl\|\sum_{n=1}^N r_n d_n \Bigr\|^p\Bigr)^\frac1p
\leq \Bigl(\E\Bigl\|\sum_{n=1}^N d_n\Bigr\|^p\Bigr)^\frac1p \leq C_p
\Bigl(\E\Bigl\|\sum_{n=1}^N r_n d_n \Bigr\|^p\Bigr)^\frac1p, \ \
N\geq 1.
\end{equation}
Here $(r_n)_{n\geq 1}$ is a Rademacher sequence independent of
$(d_n)_{n\geq 1}$. In \cite{Garling:RMTI} both inequalities in
\eqref{eq:rUMD} have been studied separately. We will consider a
different splitting of the UMD property below. For Paley-Walsh
martingales the concepts coincide as we will explain below.

\medskip

Let $X$ be a UMD Banach space and let $p\in (1, \infty)$. Let
$(d_n)_{n\geq 1}$ and $(e_n)_{n\geq 1}$ in $L^p(\O;X)$ be tangent
martingale differences, where $(e_n)_{n\geq 1}$ satisfies the (CI)
condition. In \cite{mcconnell:decoup} McConnell and independently
Hitczenko in \cite{hitczenko:notes} have proved that there exists a
constant $C=C(p,X)$ such that
\begin{equation}\label{eq:tangent}
C^{-1}\Big(\E\Big\|\sum_{n=1}^N e_n \Big\|^p\Big)^{\frac1p} \leq
\Big(\E\Big\|\sum_{n=1}^N d_n \Big\|^p\Big)^{\frac1p} \leq
C\Big(\E\Big\|\sum_{n=1}^N e_n \Big\|^p\Big)^{\frac1p},  \ \ N\geq
1.
\end{equation}
Moreover, one may take $C$ to be the UMD constant $\beta_{p,X}$. The
proof of \eqref{eq:tangent} is based on the existence of a biconcave
function for UMD spaces constructed by Burkholder in \cite{Bu2}. In
\cite{montgomery:represmart} Montgomery-Smith has found a proof
based on the definition of the UMD property. The right-hand side of
inequality \eqref{eq:tangent} also holds for $p=1$ as we will show
in Proposition \ref{prop:Xp1}.

If \eqref{eq:tangent} holds for a space $X$, then specializing to
Paley-Walsh martingales will show that $X$ has the UMD property (cf.
\cite{mcconnell:decoup}). Therefore, \eqref{eq:tangent} is naturally
restricted to the class of UMD spaces. Recall that a Paley-Walsh
martingale is a martingale that is adapted with respect to the
filtration $(\sigma(r_1, \ldots, r_n))_{n\geq 1}$, where
$(r_n)_{n\geq 1}$ is a Rademacher sequence. In this note we study
the second inequality in \eqref{eq:tangent}. This seems to be the
most interesting one for applications and we will show that it holds
for a class of Banach spaces which is strictly wider than UMD.

Let $(S,\Sigma, \mu)$ be a $\sigma$-finite measure space. We will
show that the right-hand side inequality in \eqref{eq:tangent} also
holds for $X=L^1(S)$. More generally one may take $X=L^1(S;Y)$,
where $Y$ is a UMD space (see Theorem \ref{thm:main} below). Notice
that $X$ is not a UMD space, since it is not reflexive in general.
It is not clear how to extend the proofs in \cite{hitczenko:notes,
mcconnell:decoup, montgomery:represmart} to this setting.

The right-hand side of \eqref{eq:tangent} has several applications.
For instance it may be used for developing a stochastic integration
theory in Banach spaces \cite{neervenVeraarWeis:stochIntUMD}. With
the same methods as in \cite{neervenVeraarWeis:stochIntUMD} one can
obtain sufficient conditions for stochastic integrability and
one-sided estimates for stochastic integrals for $L^1$-spaces.

Let us recall some convenient notation. For a sequence of $X$-valued
random variables $(\xi_n)_{n\geq 1}$ we will write $\xi_n^* =
\sup_{1\leq m\leq n} \|\xi_m\|$ and $\xi^* = \sup_{n\geq 1}
\|\xi_n\|$.

\section{Results\label{sec:main}}
We say that a Banach space $X$ has the {\em decoupling property for
tangent m.d.s. (martingale difference sequences)} if for all $p\in
[1, \infty)$ there exists a constant $C_{p}$ such that for all
martingales difference sequences $(d_n)_{n\geq 1}$ in $L^p(\O;X)$
and its decoupled tangent sequence $(e_n)_{n\geq 1}$ the estimate
\begin{equation}\label{eq:tangent2}
\Big(\E\Big\|\sum_{n=1}^N d_n \Big\|^p\Big)^{\frac1p} \leq
C_p\Big(\E\Big\|\sum_{n=1}^N e_n \Big\|^p\Big)^{\frac1p}, \ \ N\geq
1
\end{equation}
holds.

Let $p\in [1, \infty)$. Notice that if a martingale difference
sequence $(e_n)_{n\geq 1}$ in $L^p(\O;X)$ satisfies the (CI)
property, then
\begin{equation}\label{eq:starout}
\Big(\E\sup_{N\geq 1}\Big\|\sum_{n=1}^N e_n \Big\|^p\Big)^{\frac1p}
\eqsim \sup_{N\geq 1}\Big(\E\Big\|\sum_{n=1}^N e_n
\Big\|^p\Big)^{\frac1p}.
\end{equation}
This is well-known and easy to prove. Indeed, let
$(\tilde{e}_n)_{n\geq 1}$ be an independent copy of $(e_n)_{n\geq
1}$. Expectation with respect to $(\tilde{e}_n)_{n\geq 1}$ will be
denoted by $\tilde\E$. It follows from Jensen's inequality and the
L\'evy-Octaviani inequalities for symmetric random variables (cf.
\cite[Section 1.1]{KwWo}) applied conditionally that
\[\begin{aligned}
\Big(\E\sup_{N\geq 1}\Big\|\sum_{n=1}^N e_n \Big\|^p\Big)^{\frac1p}
&=  \Big(\E\sup_{N\geq 1}\Big\|\tilde\E\sum_{n=1}^N e_n -\tilde{e}_n
\Big\|^p\Big)^{\frac1p}
\\ & \leq \Big(\E\tilde \E\sup_{N\geq
1}\Big\|\sum_{n=1}^N e_n -\tilde{e}_n\Big\|^p\Big)^{\frac1p}
\\ & \leq 2^{\frac1p}\sup_{N\geq 1}\Big(\E\tilde \E\Big\|\sum_{n=1}^N e_n -\tilde{e}_n\Big\|^p\Big)^{\frac1p}
\\ & \leq 2^{1+\frac1p}\sup_{N\geq 1}\Big(\E\Big\|\sum_{n=1}^N
e_n\Big\|^p\Big)^{\frac1p}.
\end{aligned}
\]
Notice that Doob's inequality is only applicable for $p\in (1,
\infty)$.

\begin{proposition}\label{prop:Xp1}
If $X$ is a UMD space, then $X$ satisfies the decoupling property
for tangent m.d.s.
\end{proposition}
\begin{proof}
The case that $p\in (1, \infty)$ is already contained in
\eqref{eq:tangent}, but the case $p=1$ needs some comment. In
\cite{hitczenko:notes} it has been proved that for all $p\in
[1,\infty)$ there exists a constant $C_{p,X}$ such that for all
tangent martingale difference sequences $(d_n)_{n\geq 1}$ and
$(e_n)_{n\geq 1}$ which are conditionally symmetric one has
\begin{equation}\label{eq:condsymmcase}
C_{p,X}^{-1} \|g^*_n\|_{L^p(\O;X)} \leq \|f^*_n\|_{L^p(\O;X)} \leq
C_{p,X} \|g^*_n\|_{L^p(\O;X)}, \  n\geq 1
\end{equation}
where $f_n = \sum_{k=1}^n d_k$ and $g_n = \sum_{k=1}^n e_k$. It is
even shown that $\E\Phi(f^*_n)\leq C_{p,X,\Phi} \Phi(g^*_n)$ for
certain convex functions $\Phi$. Since \cite{hitczenko:notes} is
unpublished we briefly sketch the argument for convenience. Some
arguments are explained in more detail in the proof of Theorem
\ref{thm:chardec}.

Let $\Phi:\R_+\to \R_+$ be a continuous increasing function such
that for some $\alpha>0$, $\Phi(2t)\leq \alpha\Phi(t)$ for all
$t\geq 0$. Let $N$ be an arbitrary index. Let $(d_n)_{n\geq 1}$ and
$(e_n)_{n\geq 1}$ be conditionally symmetric and tangent martingale
difference sequences, with $d_n=e_n=0$ for $n>N$. Let $f$ and $g$ be
the corresponding martingales. By \eqref{eq:tangent} it follows that
for all $p\in (1, \infty)$,
\begin{equation}\label{eq:weaktypeumd}
\lambda \P(f_n^*\geq \lambda) \leq C_p \|g_n\|_{L^p(\O;X)}, \ \
\lambda\geq 0.
\end{equation}
Let $a_n = \max_{m<n}\{\|d_m\|, \|e_m\|\}$, $d_n' =
d_n\one_{\|d_n\|\leq 2 a_n}$, $d_n''=d_n \one_{\|d_n\|>2a_n}$,
$e_n'=e_n\one_{\|e_n\|\leq 2a_n}$, $e_n'' = e_n
\one_{\|e_n\|>2a_n}$. By the conditional symmetry, these sequences
denote martingale difference sequences. The corresponding
martingales will be denoted by $f', f'',g',g''$. Then we have
$\|d_n''\|\leq 2(a_{n+1}-a_n)$. Therefore, it follows from
$a_{N+1}=0$ and \cite[Lemma 1]{HitComp} that
\begin{equation}\label{eq:hulpHit1}
\E\Phi(f_N''^*)\leq \E\Phi\Big(\sum_{n=1}^N \|d_n''\| \Big)\leq
\alpha \E\Phi(a_N^*)\leq 2\alpha \E\Phi(e^*_N).
\end{equation}
Now for $\delta>0$, $\beta>1+\delta$, $\lambda>0$ let
\[\mu = \inf\{n\geq 0: f_n'>\lambda\}, \ \ \ \nu=\inf\{n\geq 0: f_n'>\beta\lambda\},\]
\[\sigma=\inf\{n\geq 0: g_n'>\delta\lambda \ \text{or} \ a_{n+1}>\delta\lambda\}.\]
As in \cite{Burkholder:fu} it follows from \eqref{eq:weaktypeumd}
applied to $f'$ and $g'$ and \cite[Lemma 7.1]{Burkholder:fu} that
\begin{equation}\label{eq:hulpHit2}
\E \Phi(f_N'^*)\leq c (\E\Phi(g_N'^*)+ \E\Phi(a_N^*))\leq c'
\E\Phi(g_N^*).
\end{equation}
Now \eqref{eq:condsymmcase} with $n=N$ follows from
\eqref{eq:hulpHit1} and \eqref{eq:hulpHit2} with $\Phi(x) =
\|x\|^p$.

By \eqref{eq:starout} and \eqref{eq:condsymmcase} it follows that
for all $n\geq 1$,
\[ \|f_n\|_{L^p(\O;X)} \lesssim C_{p,X} \|g_n\|_{L^p(\O;X)}, \ n\geq 1.\]
By the same symmetrization argument as in \cite[Lemma
2.1]{hitczenko:decoupineq} we obtain that for all decoupled tangent
martingale difference sequences $(d_n)_{n\geq 1}$ and $(e_n)_{n\geq
1}$ we have
\[\|f_n\|_{L^p(\O;X)} \lesssim C_{p,X} \|g_n\|_{L^p(\O;X)}, \  n\geq 1,\]
where again $f$ and $g$ are the martingales corresponding to
$(d_n)_{n\geq 1}$ and $(e_n)_{n\geq 1}$. This proves the result.
\end{proof}

Next we give a negative example.
\begin{example}\label{ex:c0}
For every $p\in [1, \infty)$ the space $c_0$ does not satisfy
\eqref{eq:tangent2}. In particular $c_0$ does not satisfy the
decoupling property for tangent m.d.s.
\end{example}
\begin{proof}
We specialize \eqref{eq:tangent2} to Paley-Walsh martingales, i.e.
$d_n = r_n f_n(r_1, \ldots, r_{n-1})$ and $e_n =  \tilde{r}_n
f_n(r_1, \ldots, r_{n-1})$, where $(r_n)_{n\geq 1}$ and
$(\tilde{r}_n)_{n\geq 1}$ are two independent Rademacher sequences
and $f_n:\{-1,1\}^{n-1}\to X$. It then follows from
\eqref{eq:tangent2} that
\[\begin{aligned}
\Big(\E\Big\|\sum_{n=1}^N r_n f_n(r_1, \ldots, r_{n-1})
\Big\|^p\Big)^{\frac1p} & \leq C\Big(\E\Big\|\sum_{n=1}^N
\tilde{r}_n f_n(r_1, \ldots, r_{n-1}) \Big\|^p\Big)^{\frac1p}
\\ & = C\Big(\E\Big\|\sum_{n=1}^N \tilde{r}_n r_n f_n(r_1,
\ldots, r_{n-1}) \Big\|^p\Big)^{\frac1p}, \ \ N\geq 1.
\end{aligned}\]
This inequality does not hold for the space $c_0$ as follows from
\cite[p.\ 105]{Garling:RMTI}.
\end{proof}

As a consequence of Example \ref{ex:c0} and the Maurey-Pisier
theorem we obtain the following result.
\begin{corollary}
If a Banach space $X$ satisfies the decoupling property for tangent
m.d.s. then it has finite cotype.
\end{corollary}

In \cite{Garling:RMTI} Garling studied both inequalities in
\eqref{eq:rUMD} separately. A space for which both inequalities of
\eqref{eq:rUMD} hold is a UMD space. Inequality \eqref{eq:tangent}
suggests another way to split the UMD property into two parts. We do
not know how the properties from \cite{Garling:RMTI} are related to
this. In the following remark we observe that they are related for
certain martingales.
\begin{remark}
\renewcommand{\labelenumi}{(\roman{enumi})}
\renewcommand{\theenumi}{(\roman{enumi})}
\
\begin{enumerate}
\item From the construction in Example \ref{ex:c0} one can see that the decoupling
property for Paley-Walsh martingales is the same property as
\begin{equation}\label{eq:GaPW}
\Big(\E\Big\|\sum_{n=1}^N d_n \Big\|^p\Big)^{\frac1p} \leq
C\Big(\E\Big\|\sum_{n=1}^N \tilde{r}_n d_n\Big\|^p\Big)^{\frac1p}
\end{equation}
from \cite{Garling:RMTI} for Paley-Walsh martingales. Here
$(d_n)_{n\geq 1}$ is a Paley-Walsh martingale difference sequence
and $(\tilde r_n)_{n\geq 1}$ is a Rademacher sequence independent
from $(d_n)_{n\geq 1}$.
\item One may also consider the relation between the first
inequality in \eqref{eq:tangent} and the reverse of estimate
\eqref{eq:GaPW}. These, too, are equivalent when restricted to
Paley-Walsh martingales. However, on the whole these inequalities
are of less interest because there are no spaces known that satisfy
them and do not satisfy the UMD property (cf.
\cite{geiss:counterexample}).
\end{enumerate}
\renewcommand{\labelenumi}{(\arabic{enumi})}
\renewcommand{\theenumi}{(\arabic{enumi})}
\end{remark}

\begin{problem}[\cite{geiss:counterexample}]
Is there a Banach space which is not UMD, but satisfies the reverse
estimate of \eqref{eq:GaPW} ?
\end{problem}
It is known that if the reverse of \eqref{eq:GaPW} holds for a
Banach space $X$, then $X$ has to be superreflexive (cf.
\cite{Garling:RMTI,geiss:counterexample}).

\begin{problem}
If \eqref{eq:GaPW} holds for all Paley-Walsh martingales, does this
imply \eqref{eq:GaPW} for arbitrary $L^p$-martingales?
\end{problem}
Recall from \cite{Bu3,maurey:BSgeometry} that for \eqref{eq:UMD}
such a result holds.

\begin{problem}
Does \eqref{eq:tangent2} for Paley-Walsh martingales (or
equivalently \eqref{eq:GaPW}) imply \eqref{eq:tangent2} for
arbitrary $L^p$-martingales?
\end{problem}
Recall from \cite{mcconnell:decoup} that this is true if one
considers \eqref{eq:tangent} instead of \eqref{eq:tangent2}.

\begin{problem}
If a Banach lattice satisfies certain convexity and smoothness
assumptions, does this imply that it satisfies the decoupling
property \eqref{eq:tangent2}?
\end{problem}
This problem should be compared with the example in
\cite{Garling:RMTI}, where Garling constructs a Banach lattice which
satisfies upper $2$ and lower $q$ estimates with $q>4$, but which
does not satisfy \eqref{eq:GaPW} for arbitrary $L^p$-martingales.

In the next theorem and remark we characterize the decoupling
property for tangent m.d.s. for a space $X$.

\begin{theorem}\label{thm:chardec}
Let $X$ be a Banach space. The following assertions are equivalent:
\begin{enumerate}
\item $X$ has the decoupling property \eqref{eq:tangent2} for tangent m.d.s.
\item There exists a constant $C$ such that for all martingales difference sequences
$(d_n)_{n\geq 1}$ in $L^1(\O;X)$ and its decoupled tangent sequence
$(e_n)_{n\geq 1}$ one has that
\[
\E\Big\|\sum_{n=1}^N d_n \Big\|\leq C\E\Big\|\sum_{n=1}^N e_n
\Big\|, \ \ N\geq 1.
\]
\item There exists a constant $C$ such that for all martingales difference sequences
$(d_n)_{n\geq 1}$ in $L^1(\O;X)$ and its decoupled tangent sequence
$(e_n)_{n\geq 1}$ one has that
\[
\lambda\P\Big(\Big\|\sum_{n=1}^N d_n \Big\|>\lambda \Big) \leq
C\E\Big\|\sum_{n=1}^N e_n \Big\|, \ \lambda\geq 0, \ N\geq 1.
\]
\end{enumerate}
\end{theorem}

Although characterizations of the above form are standard in the
context of vector valued martingales (cf.
\cite{Burkholder:Geom,Garling:RMTI}), the proof of the implication
(3) $\Rightarrow$ (1) requires some new ideas.

\begin{remark}\label{rem:pq}
\renewcommand{\labelenumi}{(\roman{enumi})}
\renewcommand{\theenumi}{(\roman{enumi})}
\
\begin{enumerate}
\item Instead of (2) one could assume that \eqref{eq:tangent2} holds
for some $p\in [1, \infty)$. Let us call this property (2)$_p$. By
the Markov inequality (2)$_p$ implies in particular that
\begin{equation}\label{eq:3prime}
\lambda^p\P\Big(\Big\|\sum_{n=1}^N d_n \Big\|>\lambda \Big) \leq
C\E\Big\|\sum_{n=1}^N e_n \Big\|^p, \ \ N\geq 1
\end{equation} which
we call (3)$_p$. We do not know whether (2)$_p$ or (3)$_p$ is
equivalent to (1). However in proof below we actually show that if
(3)$_p$ holds for some $p\in [1, \infty)$, then (2)$_q$ holds for
arbitrary $q\geq p$.
\item The statements (1), (2) and (3) of Theorem \ref{thm:chardec} are also equivalent to (1), (2) and (3)
with $\Big\|\sum_{n=1}^N d_n \Big\|$ replaced by $\sup_{N\geq
1}\Big\|\sum_{n=1}^N d_n \Big\|$ and $\Big\|\sum_{n=1}^N e_n \Big\|$
replaced by $\sup_{N\geq 1}\Big\|\sum_{n=1}^N e_n \Big\|$. This
follows from the proof below, and from \eqref{eq:starout}.
\item Condition (3) (in the form with suprema on the left-hand side) clearly implies that there exists a constant $C$ such that
 for all martingales difference sequences
$(d_n)_{n\geq 1}$ in $L^1(\O;X)$ and its decoupled tangent sequence
$(e_n)_{n\geq 1}$ one has that
\[\text{if} \ \ \sup_{N\geq 1} \Big\|\sum_{n=1}^N d_n\Big\|>1 \ \ \text{a.s. then} \ \ \E\sup_{N\geq 1}\Big\|\sum_{n=1}^N e_n\Big\|\geq C.\]
The converse holds as well as may be shown with the same argument as
in \cite[Theorem 1.1]{Burkholder:Geom}.
\end{enumerate}
\renewcommand{\labelenumi}{(\arabic{enumi})}
\renewcommand{\theenumi}{(\arabic{enumi})}
\end{remark}

\begin{problem}
Does inequality (2)$_p$ as defined in part (i) of Remark
\ref{rem:pq} imply statement (1) in Theorem \ref{thm:chardec}?
\end{problem}

\begin{proof}[Proof of Theorem \ref{thm:chardec}]
The implications (1) $\Rightarrow$ (2) $\Rightarrow$ (3) are
obvious. Therefore, we only need to show (3) $\Rightarrow$ (1). We
will actually show what is stated in Remark \ref{rem:pq}: If
\eqref{eq:3prime} holds for some $p\in [1, \infty)$, then
\eqref{eq:tangent2} holds for all $q\geq p$. This in particular
shows that (3) implies (1).

Assume that for some $p\in[1, \infty)$, \eqref{eq:3prime} holds for
all martingale difference sequences $(d_n)_{n\geq 1}$ and its
decoupled tangent sequence $(e_n)_{n\geq 1}$. Let $q\in [p, \infty)$
be arbitrary and fix an arbitrary  $X$-valued martingale difference
sequence $(d_n)_{n\geq 1}$ with its decoupled tangent sequence
$(e_n)_{n\geq 1}$. We will show that there is a constant $C$ such
that
\begin{equation}\label{eq:toshow}
\Big(\E\Big\|\sum_{n=1}^N d_n \Big\|^q\Big)^{\frac1q} \leq
C\Big(\E\Big\|\sum_{n=1}^N e_n \Big\|^q\Big)^{\frac1q}, \ \ N\geq 1.
\end{equation}
Fixing $N$, we clearly may assume that $d_n$ and $e_n$ are non-zero
only if $n\leq N$. We write $f_n = \sum_{k=1}^n d_k$, $g_n =
\sum_{k=1}^n e_k$ and $f = \limn f_n$, $g=\limn g_n$. It suffices to
show that $\|f\|_{L^q}\leq \|g\|_{L^q}$.

\smallskip
{\em Step 1. Concrete representation of decoupled tangent
sequences:}

By Montgomery-Smith's representation theorem
\cite{montgomery:represmart} we can find functions $h_n\in
L^p([0,1]^n;X)$ for $n\geq 1$ such that
\[\int_0^1 h_n(x_1, \ldots, x_n) \, dx_n = 0 \]
for almost all $x_1, \ldots, x_{n-1}$ and such that if we define
$\widehat{d}_n, \widehat{e}_n:[0,1]^{\N}\times[0,1]^{\N}\to X$ as
\[\begin{aligned}
\widehat{d}_n((x_n)_{n\geq 1}, (y_n)_{n\geq 1}) &= h_n(x_1, \ldots,
x_{n-1}, x_n)
\\ \widehat{e}_n((x_n)_{n\geq 1}, (y_n)_{n\geq 1}) &= h_n(x_1, \ldots,
x_{n-1}, y_n),
\end{aligned}\] then the sequence $(\widehat{d}_n,
\widehat{e}_n)_{n\geq 1}$ has the same law as $({d}_n, {e}_n)_{n\geq
1}$. Therefore, it suffices to show \eqref{eq:toshow} with $d_n$ and
$e_n$ replaced by $\widehat{d}_n$ and $\widehat{e}_n$. For
convenience set $h_0 = d_0 = e_0 = 0$.

For all $n\geq 1$ let $\widehat{\F}_n = \calL_n \otimes \calL_n$,
where $\calL_n$ is the minimal complete $\sigma$-algebra on
$[0,1]^{\N}$ for which the first $n$ coordinates are measurable. Let
$\widehat{\mathcal{G}} = \sigma\Big(\bigcup_{n\geq 1}\calL_n \otimes
\calL_0\Big)$. Then $(\widehat{d}_n)_{n\geq 1}$ and
$(\widehat{e}_n)_{n\geq 1}$ are $(\F_n)_{n\geq 0}$-tangent and
$(\widehat{e}_n)_{n\geq 1}$ satisfies condition (CI) with
$\widehat{\mathcal{G}}$.

We will use the above representation in the rest of the proof, but
for convenience we will leave out the hats in the notation.

\smallskip

{\em Step 2. The Davis decomposition:}

We may write $h_n = h_n^{(1)} + h_n^{(2)}$, where
$h_n^{(1)},h_n^{(2)}:[0,1]^{\N}\to X$ are given by
\[h_n^{(1)} = u_n - \E(u_n|\calL_{n-1})\]
\[h_n^{(2)} = v_n - \E(u_n|\calL_{n-1}),\]
where $u_n,v_n:[0,1]^n\to X$ are defined as
\[u_n(x_1, \ldots, x_n) =h_n(x_1, \ldots, x_n) \one_{\|h_n(x_1, \ldots,
x_n)\|\leq 2 \|h_{n-1}^*(x_1, \ldots, x_{n-1})\|}\]
\[
v_n(x_1, \ldots, x_n) =h_n(x_1, \ldots, x_n) \one_{\|h_n(x_1,
\ldots, x_n)\|> 2 \|h_{n-1}^*(x_1, \ldots, x_{n-1})\|}.
\]
Notice that for the conditional expectation $\E(u_n|\calL_{n-1})$ we
may use the representation
\[(x_m)_{m\geq 1} \mapsto \int_0^1 h_n(x_1, \ldots, x_n) \, dx_n.\]
For $i=1, 2$ define
\[d_n^{(i)}((x_n)_{n\geq 1}, (y_n)_{n\geq 1}) = h_n^{(i)}(x_1, \ldots, x_{n-1}, x_n)\]
\[e_n^{(i)}((x_n)_{n\geq 1}, (y_n)_{n\geq 1}) = h_n^{(i)}(x_1, \ldots, x_{n-1}, y_n)\]
Then for $i=1, 2$ it holds that $(d_n^{(i)})_{n\geq 1}$ and
$(e_n^{(i)})_{n\geq 1}$ are tangent and the latter satisfies
condition (CI). For $i=1, 2$ write $f_n^{(i)} = \sum_{k=1}^n
d_n^{(i)}$ and $g_n^{(i)} = \sum_{k=1}^n e_n^{(i)}$.

We will now proceed with the estimates. The first part is rather
standard, but we include it for convenience of the reader. The
second part is less standard and is given in Step 3. As in \cite[p.\
33]{Burkholder:fu} one has
\begin{equation}\label{eq:hulp1}
\sum_{n\geq 1} \|v_n\| \leq 2 \|d^*\|.
\end{equation} It follows from
\cite[Proposition 25.21]{kallenberg} that
\begin{equation}\label{eq:hulp2}
\Big\|\sum_{n\geq 1} \E(\|v_n\||\calL_{n-1})\Big\|_{L^q} \leq
q\Big\|\sum_{n\geq 1} \|v_n\|\Big\|_{L^q} \leq 2 q\|d^*\|_{L^q}.
\end{equation}
Now as in \cite[p.\ 33]{Burkholder:fu} we obtain that
\begin{equation}\label{eq:f2estd}
\begin{aligned} \|f^{(2)*}\|_{L^q}&\leq \Big\|\sum_{n\geq 1}
\|v_n\|\Big\|_{L^q} + \Big\|\sum_{n\geq 1} \|\E
(u_n|\calL_{n-1})\|\Big\|_{L^q}
\\ & \leq 2 \|d^*\|_{L^q} + \Big\|\sum_{n\geq 1} \|\E(v_n|\calL_{n-1})\|\Big\|_{L^q}
\\ & \leq (2 +2 q) \|d^*\|_{L^q},
\end{aligned}
\end{equation}
where we used \eqref{eq:hulp1}, \eqref{eq:hulp2} and
$\E(u_n|\calL_{n-1})= -\E(v_n|\calL_{n-1})$. By \cite[Theorem
5.2.1]{KwWo} and \eqref{eq:starout} 
\begin{equation}\label{eq:hulp3} \|d^*\|_{L^q}\leq
2^{\frac1q} \|e^*\|_{L^q}\leq 2^{1+\frac1q} \|g^*\|_{L^q}\leq
c_q\|g\|_{L^q},
\end{equation}
where $c_q$ is a constant. This shows that
\[ \|f^{(2)}\|_{L^q}\leq (2 +2 q) c_q \|g\|_{L^q}.\]

Next we estimate $f^{(1)}$. We claim that there exists a constant
$c_q'$ such that
\begin{equation}\label{eq:claim}
\|f^{(1)*}\|_{L^q}\leq c_q'\Big(\|g^{(1)*}\|_{L^q} +
\|d^*\|_{L^q}\Big).
\end{equation}
Let us show how the result follows from the claim before we prove
it. By \eqref{eq:hulp3} we can estimate $\|d^*\|_{L^q}$. To estimate
$\|g^{(1)*}\|_{L^q}$ we write
\[\|g^{(1)*}\|_{L^q}\leq  \|g^{(2)*}\|_{L^q} + \|g^*\|_{L^q}.\]
With the same argument as in \eqref{eq:f2estd} it follows that
\[\|g^{(2)*}\|_{L^q}\leq (2 +2 q) \|e^*\|_{L^q}\leq (4 +4 q) \|g^*\|_{L^q}.\]
Therefore, \eqref{eq:starout} gives the required estimate.

\smallskip

{\em Step 3. Proof of the claim \eqref{eq:claim}}.

For the proof of the claim we will use \cite[Lemma
7.1]{Burkholder:fu} with $\Phi(\lambda) = \lambda^q$.
To check the conditions of this lemma we will use our assumption. We
use an adaption of the argument in \cite[p.\
1000-1001]{Burkholder:Geom}.

Choose $\delta>0$, $\beta>1+\delta$ and $\lambda>0$ and define the
stopping times
\[\begin{aligned}
\mu&=\inf\{n:\|f_n^{(1)}\|>\lambda\};\\
\nu&=\inf\{n:\|f_n^{(1)}\|>\beta \lambda\};\\
\sigma&=\inf\{n: (\mathbb{E}(\|g_n^{(1)} \|^p
|\mathcal{G}))^{\inv{p}}>\delta \lambda \textrm{ or } 4
d_{n}^*>\delta\lambda \}.
\end{aligned}\]
Notice that these are all $(\calL_n)_{n\geq 1}$-stopping times. To
see this for $\sigma$, use the fact that
\[(x_m)_{m\geq 1} \mapsto \int_{[0,1]^n} \Big\|\sum_{k=1}^n h_k^{(1)}(x_1, \ldots, x_{k-1}, y_k)\Big\|^p \, dy_1, \ldots, dy_n\]
is a version for $\mathbb{E}(\|g_n^{(1)} \|^p|\mathcal{G})$ which it
is $\calL_{n-1}$-measurable, so certainly $\calL_{n}$-measurable.

Define the transforms $F$ and $G$ of $f^{(1)}$ and $g^{(1)}$ as $F_n
= \sum_{k=1}^n \one_{\{\mu<k\leq \nu\wedge \sigma\}}d_k^{(1)}$ and
$G_n = \sum_{k=1}^n \one_{\{\mu<k\leq \nu\wedge \sigma\}}e_k^{(1)}$,
for $n\geq 1$. Since $\one_{\{\mu<k\leq \nu\wedge \sigma\}}$ is
$\calL_{k-1}$-measurable it follows that $F$ and $G$ are martingales
with martingale difference sequences that are decoupled tangent
again.

Now consider $\mathbb{E}(\|G\|^p | \mathcal{G})$ on the sets
$\{\sigma\leq \mu \}$, $\{\mu<\sigma=\infty\}$ and
$\{\mu<\sigma<\infty \}$. On the first set we clearly have
$\mathbb{E}(\|G_n \|^p|\mathcal{G})=0$ for any $n\geq 1$. On the
second set we have for every $n\geq 1$
\[\begin{aligned}
(\mathbb{E}(\|G_n \|^p|\mathcal{G}))^{\inv{p}}&=(\mathbb{E}(\|
g^{(1)}_{n\land \nu} - g^{(1)}_{n\land \mu}
\|^p|\mathcal{G}))^{\inv{p}} \\ &\leq (\mathbb{E}(\| g^{(1)}_{n\land
\nu} \|^p|\mathcal{G}))^{\inv{p}} + (\mathbb{E}(\| g^{(1)}_{n\land
\mu} \|^p|\mathcal{G}))^{\inv{p}} \leq 2\delta\lambda
\end{aligned}\]
while on the set $\{\mu<\sigma<\infty \}$ we have
\[\begin{aligned}
(\mathbb{E}(\|g_n \|^p|\mathcal{G}))^{\inv{p}}&=(\mathbb{E}(\|g^{(1)}_{n\land \nu \land \sigma} - g^{(1)}_{n\land \mu} \|^p|\mathcal{G}))^{\inv{p}}  \\
& \leq (\mathbb{E}(\|e^{(1)}_{\sigma} \|^p|\mathcal{G}))^{\inv{p}} + (\mathbb{E}(\|g^{(1)}_{n\land \nu \land (\sigma-1)} \|^p|\mathcal{G}))^{\inv{p}} + (\mathbb{E}(\|g^{(1)}_{n\land \mu} \|^p|\mathcal{G}))^{\inv{p}}\\
&\leq (\mathbb{E}(\|e^{(1)}_{\sigma} \|^p|\mathcal{G}))^{\inv{p}} +
2\delta\lambda.
\end{aligned}\]
Since the difference sequences of $f^{(1)}$ and $g^{(1)}$ are
tangent and the difference sequence of $g^{(1)}$ satisfies the (CI)
condition we have
\[\begin{aligned}
\mathbb{E}\left(\|e^{(1)}_{\sigma} \|^p|\mathcal{G}\right)
&=\mathbb{E}\left(\sum_{n=1}^{\infty}\|e^{(1)}_{n}
\|^p1_{\{\sigma=n\}} |\mathcal{G}\right)
= \sum_{n=1}^{\infty}\mathbb{E}(\| e^{(1)}_{n} \|^p|\mathcal{G}) 1_{\{\sigma=n\}}\\
&= \sum_{n=1}^{\infty}\mathbb{E}(\|e^{(1)}_{n} \|^p|\mathcal{F}_{n-1}) 1_{\{\sigma=n\}} \\
&= \sum_{n=1}^{\infty}\mathbb{E}(\|d^{(1)}_{n}
\|^p|\mathcal{F}_{n-1}) 1_{\{\sigma=n\}} \leq 4^p
\sum_{n=1}^{\infty}(d_{n-1}^*)^{p} 1_{\{\sigma=n\}}\leq
(\delta\lambda)^{p}.
\end{aligned}\]
Here we used that from Davis decomposition we know that $4
d_{n-1}^*$ is an $\F_{n-1}$-measurable majorant for $\|d_n^{(1)}\|$.

On the whole we have
\[\begin{aligned}
(\mathbb{E}(\|G_n \|^p|\mathcal{G}))^{\inv{p}}\leq
3\delta\lambda1_{\{\mu<\infty
\}}=3\delta\lambda1_{\{f^{(1)*}>\lambda \}},
\end{aligned}\]
hence
\begin{equation}\label{ineqGn}
\mathbb{E}\|G \|^p\leq 3^p\delta^p\P\{f^{(1)*}>\lambda \}.
\end{equation}

Observe that on the set
\[\begin{aligned}
\{f^{(1)*}>\beta\lambda, \mathbb{E}(\|g^{(1)} \|^p | \mathcal{G})^*
\lor 4d^*
 < \delta\lambda\}
\end{aligned}\]
one has $\mu<\nu<\infty$ and $\sigma=\infty$ and therefore
\[\begin{aligned}
\| F \| = \| f_{\nu}^{(1)} -d_{\mu}^{(1)}-f_{\mu-1}^{(1)}\| \geq
\|f_{\nu}^{(1)} \| - \|d_{\mu}^{(1)} \| -\|f_{\mu-1}^{(1)} \| >
(\beta-\delta-1)\lambda.
\end{aligned}\]
Now by the assumption, applied to $F$ and $G$, and by \eqref{ineqGn}
we obtain
\[\begin{aligned}
&\mathbb{P}\{ f^{(1)*}>\beta\lambda, (\mathbb{E}(\|g^{(1)} \|^p |
\mathcal{G}))^{\inv{p}*} \lor 4d^*  < \delta\lambda\}
\leq \mathbb{P}\{\mu<\nu,\sigma=\infty\}\\
& \leq \mathbb{P}\{\|F\|> (\beta-\delta-1)\lambda  \} 
\leq C^p(\beta-\delta-1)^{-p}\lambda^{-p}\|G\|_{L^p}^{p}\\
&\leq 3^{p}C^p \delta(\beta-\delta-1)^{-p}
\mathbb{P}\{f^{(1)*}>\lambda \}.
\end{aligned}\]
Applying \cite[Lemma 7.1]{Burkholder:fu} with
$\Phi(\lambda)=\lambda^q$ gives some constant $C_q$ depending on
$C$, $p$ and $q$ such that
\[\begin{aligned}
\| f^{(1)*} \|_{L^q} &\leq C_q\big\| (\mathbb{E}(\| g^{(1)} \|^p |
\mathcal{G}))^{\inv{p}*} \lor 4d^* \big\|_{L^q} \\ & \leq
4C_q\big(\big\| (\mathbb{E}(\| g^{(1)} \|^p |
\mathcal{G}))^{\inv{p}*} \big\|_{L^q} + \| d^* \|_{L^q}\big).
\end{aligned}\]
Since $q\geq p$, \eqref{eq:claim} follows.
\end{proof}

In the above proof we have showed that Theorem \ref{thm:chardec} (2)
implies \eqref{eq:tangent2} for all $p\in [1, \infty)$ with a
constant $C_p$ with $\lim_{p\to \infty}C_p=\infty$. Using the
representation of Step 1 of the proof of Theorem \ref{thm:chardec}
one easily sees that \eqref{eq:tangent2} holds for $p=\infty$ with
constant $1$ for arbitrary Banach spaces. It is therefore natural to
consider the following problem which has been solved positively by
Hitczenko \cite{hitczenko:decoupineq} in the case that $X=\R$.

\begin{problem}
If $X$ satisfies the decoupling property, does $X$ satisfy
\eqref{eq:tangent2} with a constant $C$ independent of $p\in [1,
\infty)$?
\end{problem}

We have already observed that all UMD spaces satisfy the decoupling
inequality, thus for example the $L^p$-spaces do so for
$p\in(1,\infty)$. The next theorem states that $L^1$-spaces, which
are not UMD, satisfy the decoupling property as well.

\begin{theorem}\label{thm:main}
Let $(S, \Sigma, \mu)$ be a $\sigma$-finite measure space and let
$p\in [1, \infty)$. Let $Y$ be a UMD space and let $X=L^1(S;Y)$.
Then $X$ satisfies the decoupling property for tangent m.d.s.
\end{theorem}
The proof is based on Theorem \ref{thm:chardec} and the following
lemma which readily follows from Fubini's theorem.
\begin{lemma}\label{lem:LpX}
Let $X$ be a Banach space and let $p\in [1, \infty)$. Let $(S,
\Sigma, \mu)$ be a $\sigma$-finite measure space. If $X$ satisfies
\eqref{eq:tangent2}, then $L^p(S;X)$ satisfies \eqref{eq:tangent2}.
\end{lemma}
\begin{proof}
Let $(d_n)_{n\geq 1}$ and $(e_n)_{n\geq 1}$ be decoupled tangent
sequences in $L^p(\O;L^p(S;X))$. By Fubini's theorem there exists a
sequence $(\tilde{d}_n)_{n\geq 1}$ of functions from $\O\times S$ to
$X$ such that for almost all $\omega\in \O$, for almost all $s\in
S$, for all $n\geq 1$ we have
\[d_n(\omega)(s)  = \tilde{d}_n(\omega,s)\]
and for almost all $s\in S$, $\tilde{d}_n(s)_{n\geq 1}$ is
$\mathcal{F}_{n}$-measurable. We claim that for almost all $s\in S$,
\[\E(\tilde{d}_n(\cdot,s)|\mathcal{F}_{n-1}) =0 \ \text{a.s.}\]
To prove this it suffices to note that for all $A\in \Sigma$ and
$B\in \F_{n-1}$,
\[\int_{A} \int_B \tilde{d}_n(\omega,s) \, dP(\omega) \, d\mu(s) = \int_{A} \int_B d_n(\omega)(s) \, dP(\omega) \, d\mu(s)=0.\]
Also such $(\tilde{e}_n)_{n\geq 1}$ exists for $(e_n)_{n\geq 1}$.
Next we claim that for almost all $s\in S$, $(\tilde d_n(\cdot,
s))_{n\geq 1}$ and $(\tilde e_n(\cdot, s))_{n\geq 1}$ are tangent
and $(\tilde e_n(\cdot, s))_{n\geq 1}$ satisfies condition (CI).
Indeed, for $A$ and $B$ as before and for a Borel set $C\subset X$
we have
\[\begin{aligned}
\int_{A} \int_B \one_{\{\tilde{d}_n(\omega,s)\in C\}} \, dP(\omega)
\, d\mu(s) &= \int_{A} \int_B \one_{\{d_n(\omega)(s)\in C\}} \,
dP(\omega) \, d\mu(s)
\\ &= \int_{A} \int_B \one_{\{e_n(\omega)(s)\in C\}} \, dP(\omega) \, d\mu(s)
\\ &= \int_{A} \int_B \one_{\{\tilde{e}_n(\omega,s)\in C\}} \, dP(\omega) \,
d\mu(s).
\end{aligned}\]
This clearly suffices. Similarly, one can prove the (CI) condition.

Now by Fubini's theorem and the assumption applied for almost all
$s\in S$ we obtain that
\[
\begin{aligned}
\E\Big\|\sum_{n=1}^N d_n \Big\|^p_{L^p(S;X)}& = \int_S
\int_{\O}\Big\|\sum_{n=1}^N \tilde{d}_n(\omega, s) \Big\|^p \,
d\P(\omega) \, d\mu(s) \\ & \leq C^p \int_S
\int_{\O}\Big\|\sum_{n=1}^N \tilde{e}_n(\omega, s) \Big\|^p \,
d\P(\omega) \, d\mu(s) & = \E\Big\|\sum_{n=1}^N e_n
\Big\|^p_{L^p(S;X)}.
\end{aligned}\]
\end{proof}

\begin{proof}[Proof of Theorem \ref{thm:main}]
By Proposition \ref{prop:Xp1} the space $Y$ satisfies the decoupling
property. Therefore, we obtain from Lemma \ref{lem:LpX} that $X =
L^1(S;Y)$ satisfies \eqref{eq:tangent2} for $p=1$. Now Theorem
\ref{thm:chardec} implies that $X$ satisfies the decoupling
property.
\end{proof}

For $p\in [1, \infty)$ let $\mathcal{S}_p$ be the Schatten class of
operators on a infinite dimensional Hilbert space. For every $p\in
(1, \infty)$, $\mathcal{S}_p$ is a UMD space. Therefore, by
Proposition \ref{prop:Xp1} it satisfies the decoupling property.
Since $\mathcal{S}_1$ is the non-commutative analogue of $L^1$, it
seems reasonable to state the following problem.

\begin{problem}
Does the Schatten class $\mathcal{S}_1$ satisfy the decoupling
property \eqref{eq:tangent2}?
\end{problem}


{\em Acknowledgment} -- The authors thank Jan van Neerven for
helpful comments.

\providecommand{\bysame}{\leavevmode\hbox
to3em{\hrulefill}\thinspace}

\end{document}